\documentclass[a4paper,11pt]{scrartcl}
\usepackage{subfiles}
\usepackage[bookmarks,pagebackref,colorlinks=true,linkcolor=blue,urlcolor=blue,citecolor=blue]{hyperref}
\usepackage[alphabetic]{amsrefs}

\usepackage[utf8]{inputenc}
\usepackage[english]{babel}
\usepackage{eulervm}
\usepackage{amssymb}
\usepackage{amsmath}
\usepackage{latexsym}
\usepackage{amsthm}
\usepackage{eucal}
\usepackage{bbm}
\usepackage{chngcntr}
\usepackage{apptools}
\usepackage{mathtools}
\usepackage{authblk}
\usepackage[pdflatex]{crop}
\usepackage{tikz}
\usepackage{tikz-cd}
\usetikzlibrary{3d}
\usepackage{tikz-3dplot}
\usepackage{pgfplots}
\usepackage{subcaption}
\usepackage{microtype}
\usepackage{enumitem}
\usepackage{graphicx}
\usepackage{mathrsfs}
\usepackage[colorinlistoftodos]{todonotes}
\usepackage{yfonts}
\usepackage{tabularx}
\usepackage{color}

\allowdisplaybreaks

\setkomafont{disposition}{\normalfont\bfseries}
\numberwithin{equation}{section}

\theoremstyle{definition}
\newtheorem{thm}{Theorem}[section]
\AtAppendix{\counterwithin{thm}{section}}
\newtheorem{defn}[thm]{Definition}
\newtheorem{lemma}[thm]{Lemma}
\newtheorem{prop}[thm]{Proposition}
\newtheorem{cor}[thm]{Corollary}
\newtheorem{con}[thm]{Conjecture}\newtheorem{prob}[thm]{Problem}

\newtheorem{ex}[thm]{Example}
\newtheorem{remark}[thm]{Remark}

\newtheorem{innercustomthm}{Theorem}
\newenvironment{introthm}[1]
  {\renewcommand\theinnercustomthm{#1}\innercustomthm}
  {\endinnercustomthm}

\theoremstyle{definition}

\newcommand{\R}{\mathbb{R}}

\let\originalleft\left
\let\originalright\right
\renewcommand{\left}{\mathopen{}\mathclose\bgroup\originalleft}
\renewcommand{\right}{\aftergroup\egroup\originalright}

\usepackage{setspace}

\newcommand{\auth}[0]{{Renan Assimos, Bal\'azs M\'ark B\'ek\'esi} and Giuseppe Gentile}
\newcommand{\tit}[0]{{Remarks on the generalised Calabi-Yau problem in higher codimension}}
\newcommand{\kw}[0]{{Harmonic maps, minimal submanifolds, convexity, pullvexity, Omori-Yau maximum principle, stochastic completeness, Calabi-Yau conjecture.}}

\hypersetup{
pdfauthor={\auth},%
pdftitle={\tit},%
colorlinks, linktocpage=true, pdfstartpage=1, pdfstartview=FitV,%
breaklinks=true, pdfpagemode=UseNone, pageanchor=true, pdfpagemode=UseOutlines,%
plainpages=false, bookmarksnumbered, bookmarksopen=true, bookmarksopenlevel=1,%
hypertexnames=true, pdfhighlight=/O,%
urlcolor=blue, linkcolor=blue, citecolor=blue, 
}
\title{\tit}
\author{\auth\thanks{Correspondence: \href{mailto:renan.assimos@math.uni-hannover.de}{renan.assimos@math.uni-hannover.de}, \href{https://bmbekesi.github.io/ }{balazs.bekesi@math.uni-hannover.de}, \href{mailto:giuseppe.gentile@math.uni-hannover.de}{giuseppe.gentile@math.uni-hannover.de}}}
\affil{\small Leibniz University Hannover, Welfengarten 1, 30167 Hannover, Germany}
\date{}

\pgfplotsset{compat=1.16}
\begin{document}

\maketitle

\begin{abstract}\vspace{-1.5cm}
By introducing a more flexible notion of convexity, we obtain a new Omori-Yau maximum principle for harmonic maps. In the spirit of the Calabi-Yau conjectures, this principle is more suitable for studying the unboundedness of certain totally geodesic projections of minimal submanifolds of higher codimension. We further explore this maximum principle by applying it to conformal maps, harmonic maps into Cartan-Hadamard manifolds, as well as cone, wedge and halfspace theorems.
\end{abstract}

\begin{spacing}{0.85}
\tableofcontents
\end{spacing}



\section{Introduction}

In 1966, Calabi (cf. \cite{Calabi}) posed two conjectures about minimal hypersurfaces in $\R^n$, which we state here in a slightly more general form.
\begin{enumerate}
    \item A minimal submanifold of codimension one in $\R^n$ is unbounded.
    \item Let $P$ be an orthogonal projection of rank $k$ on $\R^n$ and consider a non-planar minimal submanifold $M$ of $\R^n$. If $k=n-2$ and the submanifold is of codimension one, then the projection of $M$ via $P$ is unbounded. 
\end{enumerate}

With the stated generality, the conjectures turn out to be false.
The first counterexample is for (b), usually referred to as the “more ambitious conjecture”. Such a counterexample was obtained by Jorge and Xavier; in \cite{jorgexavier} the authors construct a complete immersed minimal surface between two planes in $\R^3$.
Regarding (a), an astonishing counterexample has been provided by Nadirashvili in \cite{Nadirashvili} where a complete minimal disk inside a ball in $\R^3$ is constructed.
A crucial step in the analysis of counterexamples to the Calabi conjectures is provided by Mart\'in and Morales in \cite{martin2004asymptotic}. 
There the authors present a counterexample to a variation to the Calabi conjecture (a), asserting the non-existence of properly immersed complete minimal surfaces inside a ball in $\R^3$.
Indeed, Mart\'in and Morales develop new machinery to strengthen Nadirashvili's construction, obtaining a properly immersed disk. 
It is important to note that the notion of properness is refined, in the sense that the disk is properly immersed with respect to the induced topology on the ball. Despite the aforementioned counterexamples, the two Calabi conjectures led the community to a better understanding of the properties of minimal immersions in terms of embeddedness and properness.
This deepened understanding brought to a proof of both conjectures for complete embedded minimal surfaces of finite topology in $\R^3$ due to Colding and Minicozzi\cite{ColdingMinicozzi}.

\medskip

The literature covering the Calabi conjectures for surfaces in $\R^3$ is quite extensive; so we refer the interested reader to a comprehensive and up to date introduction from Meeks III, Per\'ez and Ros~\cite{Meeks2021embedded}.
We take the opportunity to also highlight the beautiful result they obtain in this paper stating that a complete embedded minimal surface of finite genus with compact boundary and exactly one limit end must be proper.
Such a result is a first, but substantial, step towards a more general Calabi type conjecture for embedded surfaces.

\begin{con}[Meeks-Per\'ez-Ros, \cite{Meeks2021embedded}]
Every connected, complete embedded minimal surface $M \subset \R^3$ of finite genus and compact (possibly empty) boundary is properly embedded in $\R^3$.    
\end{con}

The motivation for the present work was rather ambitious and fundamentally different compared to the case of surfaces in 3-dimensional Euclidean space. 
The authors were trying to understand which other instances of the Calabi conjecture could hold in higher dimensions and codimension. Therefore, the following \textbf{generalised Calabi-Yau problem} can be seen as the starting point of our investigation.

\begin{prob}[Generalised Calabi-Yau problem]\label{prob:genCalabiYau}
    Let $M$ be a non-compact minimally immersed submanifold of a manifold $N$ and let $P$ be a totally geodesic map from $N$ into another manifold $B$. Is the projection of $M$ via $P$ unbounded?
\end{prob}

The problem has the Calabi conjectures as a special case. Conjecture (a) is the case with $N = B = \R^n$ and $P=\mathrm{Id}$. For conjecture (b), the totally geodesic map $P$ is an orthogonal projection of rank $(n-2)$. 

\medskip

Colding and Minicozzi proved the Calabi-Yau conjectures as a consequence of a halfspace theorem. In particular, they show that complete minimally embedded surfaces of finite topology are proper. This evidentiates the strength of their result. Our attempt to the Calabi-Yau conjectures goes in higher codimension, where it is potentially not true that embeddedness implies properness. Therefore, it is unknown if Colding-Minicozzi's result still holds in this setting.
Compared to them, we deal with a more general class of submanifolds, indeed we do not require embeddedness or completeness. On the other hand, we pay the price of having to assume something extra. This extra could be comparable to the conclusion of Colding-Minicozzi and it is \textit{stochastic completeness}.

\medskip
The interplay between stochastic completeness and minimal immersions traces back to \cite{Omori} via a new maximum principle for non-compact manifolds. Indeed, Omori provides a maximum principle under a uniform lower bound on the sectional curvature. The next notable achievement appeared in the groundbreaking work of Yau \cite{Yau}. In there, the lower bound was weakened to a uniform lower bound on the Ricci curvature and a weaker maximum principle, compared to Omori's result, was obtained. Subsequently, the maximum principle of Yau was promoted to a definition and bears the names of Omori and Yau. By building on the foundational work of Grigori'yan \cite{Grigoryan} a spectacular link between stochastic completeness and the Omori-Yau maximum principle was established by Pigola, Rigoli and Setti \cite{PigolaRigoliSetti}. Since stochastic completeness is a property of the Brownian motion on the Riemannian manifold, the aforementioned result displays a fascinating bridge between stochastic and geometric analysis.

\medskip

The applications of the Omori-Yau maximum principle in the analysis of geometric problems are immeasurable. Amongst the various contributions, we would like to highlight some interesting developments. For instance, in \cite{maririgolicone} the authors provide non-existence results for maps into Euclidean cones depending on various parameters. A bi-halfspace theorem for translating solitons of the mean curvature flow is obtained in \cite{niels}. Further halfspace and intersection results for stochastically complete minimal and CMC surfaces in $\R^3$ appear in \cite{Leandro}. In \cite{theG} stochastic completeness plays a major role in proving the long-time existence of the prescribed mean curvature flow in a special class of Lorentzian manifolds. 

\medskip

Most of the previous results, even the classical papers of Omori and Yau, only consider functions on the underlying Riemannian manifold, i.e. the target of their maps is $\R$. Thus, the following natural question arises:
\begin{quote}
    Let $M$ be a manifold satisfying the Omori-Yau maximum principle. What can we conclude for maps $u:M\rightarrow N$?
\end{quote}

The aforementioned question combined with the generalised Calabi-Yau problem is the genesis of the ideas of this paper.

In a previous work \cite{assimosJost} a foliated version of Sampson's maximum principle \cite{sampson} has been introduced. Subsequently, in \cite{ABG} the authors generalised the maximum principle of \cite{assimosJost} to the non-compact case. As an application, non-existence results for proper harmonic maps from possibly non-compact manifolds into Riemannian cones were obtained. 

The starting point of our work was to replace the classical maximum principle in the proof of Sampson by the one of Omori and Yau. As it will be clear later this is the core technical idea driving the work.

\subsection{Outline and summary of the main results}

In Section 2 we introduce the notion of pullback convexity (Definition \ref{def:pullvexity}), which is a generalised notion of convexity. The basic idea is that convexity of $f:N \rightarrow \R$ should be adapted to the map $u:M\rightarrow N$ via the condition $tr_g(u^*\mathrm{Hess}\,f)\geq 0$. The analysis of this condition allows to conclude results about $\Delta_M(f\circ u)$ for $u$ being a harmonic map. In a forthcoming paper, we will investigate non-harmonic maps in combination with this flexible notion of convexity.

Then, we investigate some basic properties of pullvexity and relate them to familiar notions of convexity, mean-convexity, $k$-convexity and $k$-mean convexity. We conclude section 2 by explaining a relation arising for the second fundamental form of a pullvex hypersurfaces.

\medskip

We begin Section 3 by recalling the Omori-Yau maximum principle (Definition \ref{def:OY}) and proceed to prove a maximum principle for harmonic maps on manifolds satisfying the Omori-Yau maximum principle. Namely:

\begin{introthm}{\ref{thm:HOY}}[Omori-Yau for harmonic maps]
Let $(M,g)$ be stochastically complete and $u:(M,g) \rightarrow (N,h)$ be a harmonic map.
\begin{enumerate}
    \item For any strongly pullvex function $f:N \rightarrow \R$ the composition $f\circ u$ is subharmonic on $M$ and unbounded from above.
    \item Moreover, if there is a $r < \sup(f \circ u)$, then the pullvexity only needs to hold on $f^{-1}[r,\infty)$.
\end{enumerate}
\end{introthm}

By means of Theorem \ref{thm:HOY}, we recover several classical statements like a pullvex version of Sampson's maximum principle (Proposition \ref{prop:pullvex_sampson}).

\medskip

The Calabi conjectures for stochastically complete manifolds will be a corollary of the next main theorem describing geometric ways of constructing strongly pullvex functions.

\begin{introthm}{\ref{thm:tomography}}[Tomography theorem]
    Let $(M,g)$ be stochastically complete and consider the smooth maps $u:(M,g) \rightarrow (N,h)$, $P:(N,h) \rightarrow (B,\langle \cdot, \cdot \rangle)$ and $\eta:(B,\langle \cdot, \cdot \rangle)\rightarrow \R$ between Riemannian manifolds. If one of the cases
    \begin{enumerate}
    \item \begin{enumerate}[label=(\roman*)]
            \item $u$ is harmonic and non-singular at infinity,
            \item either $P$ is totally geodesic and $k$-wide with $k \geq \mathrm{rk}\,u$ or 
            \item[(ii')] $P$ is totally geodesic, strongly horizontally conformal with the rank condition $\mathrm{rk}\, u + \mathrm{rk}\, P - \dim N \geq 1$ and
            \item $\eta$ is strongly convex
        \end{enumerate}
    \item \begin{enumerate}[label=(\roman*)]
            \item $u$ is harmonic and strongly (horizontally) conformal,
            \item $P$ is totally geodesic and
            \item $\eta$ is strongly $k$-mean pullvex with respect to $P$ for $k \geq \mathrm{rk}\,u$
        \end{enumerate}
    \item \begin{enumerate}[label=(\roman*)]
            \item $u$ is harmonic and strongly (horizontally) conformal,
            \item $P$ is totally geodesic and strongly horizontally conformal,
            \item the dimensional condition $\mathrm{rk}\, u + \mathrm{rk}\, P - \dim N \geq k$ and
            \item $\eta$ is strongly $k$-mean convex
        \end{enumerate}
    \end{enumerate}
    holds, then $f:= \eta \circ P$ is strongly pullvex with respect to $u$ and $f\circ u$ is unbounded.
\end{introthm}

The previous theorem settles multiple cases of Problem \ref{prob:genCalabiYau}.

\begin{introthm}{\ref{thm:calabi}}[Calabi conjectures]
    Let $M$ be a stochastically complete manifold and $u:M \rightarrow \R^n$ a minimal immersion.
    \begin{enumerate}
        \item The image of $u$ is unbounded.
        \item Let $P:\R^n \rightarrow \R^n$ be an orthogonal projection. If $\mathrm{rk}\,u + \mathrm{rk}\, P - n \geq 1$, then $P \circ u$ is unbounded.
    \end{enumerate}
\end{introthm}

Note that if $u(M)$ is a hypersurface, $n\geq 4$, and $P$ a projection of rank $(n-2)$, (b) is the “more ambitious” Calabi conjecture. 
Another consequence is an analogue of the perturbed cone theorem appearing in \cite{ABG}.

\begin{introthm}{\ref{thm:perturbedwegde}}[Perturbed wedge theorem]
    Let $(M,g)$ be a stochastically complete manifold and $u:M \rightarrow \R^n$ a harmonic map that is non-singular at infinity. The image of $u$ cannot be contained in a wedge region of a perturbed wedge of type $(n,k)$ with $\dim M + k - n \geq 1$.
\end{introthm}

Thereafter, in Theorem \ref{thm:simplextrap} a local version of the wedge theorem is proven. Lastly, the condition of parabolicity is touched upon for concluding a halfspace result (Proposition \ref{prop:parabolichalfspace}).

\subsection{Notation and preliminaries}

\paragraph{Notation and conventions}
$(M,g)$ and $(N,h)$ will denote Riemannian manifolds without boundary of dimension $m$ and $n$ respectively, where $M$ is assumed to be the domain and $N$ is the target of maps. We will always assume that $M$ and $N$ are connected. The geodesic completeness of $M$ and $N$ is not a priori assumed. Smooth maps between $M$ and $N$ are referred to as $u:M \rightarrow N$ and smooth functions on $N$ are mostly denoted by $f$.
\paragraph{Harmonic maps}
Let $\nabla^M$ and $\nabla^N$ be the Levi-Civita connections on $M$ and $N$ respectively. For a smooth map $u:M \rightarrow N$ we can form the connection $\nabla = (\nabla^M)^* \otimes u^*\nabla^N$ on $T^*M \otimes u^*TN$. If $f:M \rightarrow \R$ is a function then $\nabla df$ is denoted by $\mathrm{Hess}\, f$. The \textbf{tension field} of $u$ is defined as $\Delta u = \mathrm{tr}_g(\nabla du)$ and $u$ is \textbf{harmonic} if $\Delta u = 0$. For $u$ being a function, the tension field of $u$ is the usual Laplace-Beltrami operator $\Delta_Mu$. An equivalent formulation of harmonicity is given as the (compact) variation of the \textbf{Dirichlet energy} $E[u] = \frac{1}{2}\int_M e(u) \mathrm{dVol}_g$ where $e(u):= \|du\|^2 :=\mathrm{tr}_g(u^*h)$ is the \textbf{energy density} of $u$. We refer the reader to \cite{eellslemaire} for more on the theory of harmonic maps.

\medskip

Some examples of harmonic maps are geodesics, minimal submanifolds, holomorphic maps between K\"ahler manifolds or \textbf{totally geodesics maps}, i.e. maps satisfying $\nabla du = 0$.

\medskip
The following proposition consists of two composition formulas and lies, in some sense, at the heart of our whole work.
\begin{prop}[Proposition 2.20 in \cite{eellslemaire}]\label{prop:compositionformula}
    Let $u:(M,g) \rightarrow (N,h)$ and $P:(N,h) \rightarrow (B,\langle \cdot, \cdot \rangle)$ be smooth maps between Riemannian manifolds. Then the composition formulas
    \begin{enumerate}
        \item $\nabla d(P\circ u) = dP(\nabla du) + u^*\nabla dP$ and
        \item $\Delta (P \circ u) = dP(\Delta u) + \mathrm{tr}_g(u^*\nabla dP)$
    \end{enumerate}
    hold.
\end{prop}

\paragraph{Acknowledgement} G. Gentile would like to thank Luciano Mari and Stefano Pigola for helpful discussions. This article is part of the second author's PhD thesis.

\section{Flexible notions of convexity}
\subsection{First properties of pullback convexity}
Convexity has been a crucial concept when regarding the non-existence results of harmonic maps. In this section, we will generalise the Hessian characterisation of convexity to a larger class of functions. 
\paragraph{Pullback convexity} The notion of pullback convexity, for short also to be called pullvexity, can be regarded as a pulled-back version of convexity, as the name suggests, and hence the etymology of our abbreviation pullvex consists of the blending of the words pullback and convex. The origin of pullback convexity arose while investigating the crucial role that is played by the composition $f \circ u$ of a harmonic map $u:M \rightarrow N$ with a function $f:N \rightarrow \R$ being subharmonic. Indeed, by means of the composition formula, one gets
$$
\Delta_M(f\circ u) = df(\Delta u) + \mathrm{tr}_{g}(u^* \nabla df) = \mathrm{tr}_{g}(u^* \nabla df) \geq 0.
$$

This observation motivates the following definition.

\begin{defn}\label{def:pullvexity}
Let $\mathcal{F} = \{u_{\alpha}:(M_\alpha,g_\alpha) \rightarrow N\}_{\alpha\in I}$ be a family of smooth maps from (possibly different) Riemannian manifolds $(M_\alpha,g_\alpha)$ into $N$. A smooth function $f:N \rightarrow \R$ is said to be \textbf{pullback convex}, or simply \textbf{pullvex}, with respect to the family of smooth maps $\mathcal{F}$ if
$$
\mathrm{tr}_{g_\alpha}(u_\alpha^* \mathrm{Hess}\, f) \geq 0
$$
for all $\alpha \in I$. Replacing $\geq 0$ by $> 0$ or $\geq C>0$ defines the notions of \textbf{strictly pullvex} and \textbf{strongly pullvex} functions, respectively. If the family $\mathcal{F}$ only consists of one map $u$, then $f$ is referred to be pullvex with respect to $u$.
\end{defn}

\paragraph{Singular values and pullback convexity} 
\begin{figure}[h]
    \centering
    \begin{tikzpicture}[scale=1.5]     
        \draw[->] (-1,0) -- (0,0) node[right] {$v_1$};
        \draw[->] (-1,0) -- (-1,1) node[above] {$v_2$}; 
        \draw[dashed] (-1,0) ellipse (1cm and 1cm);

        \draw[->, thick] (1,0) -- (2,0) node[above left] {$D_pu$};
        
        \draw[->] (4,0) -- (5.05,1.05) node[right] {$\lambda_1w_1$}; 
        \draw[->] (4,0) -- (3.5,0.5) node[above left] {$\lambda_2w_2$};
         \draw[dashed,rotate around={-45:(4,0)}] (4,0) ellipse (0.75cm and 1.5cm);
    \end{tikzpicture}
    \caption{Principal axis and singular values}
    \label{fig:singularvalues}
\end{figure}
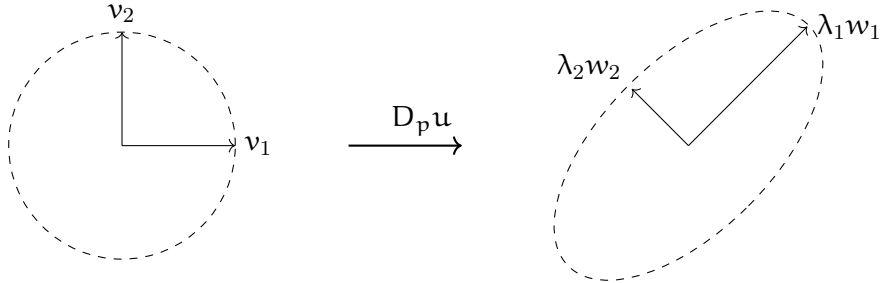

Let $u:(M,g) \rightarrow (N,h)$ be a smooth map between Riemannian manifolds. By using both metrics, we can consider the composition
$$
TM \stackrel{Du}{\longrightarrow} TN \stackrel{\flat_h}{\longrightarrow} T^*N \stackrel{{}^tDu}{\longrightarrow} T^*M \stackrel{\sharp_g}{\longrightarrow} TM
$$
with $\flat_h$ and $\sharp_g$ referring to the usual musical isomorphisms and $(Du)^* = \sharp_g\circ {}^tDu\circ \flat_h$ being the adjoint of $Du$. The eigenvalues of $(Du)^*\circ Du$ are non-negative and denoted by $\lambda_1^2,\dots,\lambda_k^2$ with $k = \min\{\dim M, \dim N\}$. The non-negative square roots $\lambda_1,\dots,\lambda_k$ are referred to as \textbf{singular values}. At each point $p\in M$ there are orthonormal vectors $v_1,\dots,v_m$ of $T_pM$ and $w_1,\dots,w_n$ of $T_{u(p)}N$ such that
$$
Du(v_i) = \lambda_i w_i \text{ and } (Du)^*(w_i)= \lambda_i v_i
$$
for $i = 1,\dots,k$. Geometrically, singular values are nothing else but the lengths of the principal axis of the ellipse in the tangent space $T_{u(p)}N$ obtained by mapping the unit sphere in $T_pM$ via $D_pu$. See Figure \ref{fig:singularvalues}.

\begin{lemma}\label{lem:pullvexlocally}
Let $f:N \rightarrow \R$ be pullvex with respect to $u:M\rightarrow N$. Locally, a pullvex function is of the form $$\mathrm{tr}_g(u^*\mathrm{Hess}\, f) = \sum_{i=1}^k \lambda_i^2 \mathrm{Hess}\, f (w_i,w_i)$$ with the singular values $\lambda_i$ together with the singular vectors $v_i, w_i$. If $u:M \rightarrow N$ is locally constant in a region $U \subseteq M$ there are neither strictly nor strongly pullvex functions.
\end{lemma}
\begin{proof}
The first formula follows by
$$\mathrm{tr}_g(u^*\mathrm{Hess}\,f) = \sum_{i=1}^m \mathrm{Hess}\,f(du(v_i),du(v_i)) = \sum_{i=1}^k \lambda_i^2\mathrm{Hess}\,f(w_i,w_i).$$The second claim follows by the singular values of $u$ vanishing on $U$.
\end{proof}
The first statement of the previous lemma shows that the notion of pullvexity can be seen as a weighted (mean-)convexity, where the weights are provided by the singular values of the map $u$. In the more general setting where the family $\mathcal{F}$ consists of multiple functions, there is an inequality for each member of the family which needs to be satisfied by the pullvex function.

\subsection{Recovering notions of convexity}
Now, we will demonstrate how pullvexity incorporates several well-known definitions of convexity.

\paragraph{Modes of convexity}
As a reminder, we will recall some well-known definitions of convexity of Riemannian manifolds. The following notions of convexity are explored in the papers \cite{BishopOneill}, \cite{HarveyLawson} and \cite{GromovSign}.
\begin{defn}
    Let $(N,h)$ be a Riemannian manifold and $f:N \rightarrow \R$ a smooth function. Let $k=1,\dots, n =\dim N$.
    \begin{enumerate}
        \item $f$ is \textbf{$k$-mean convex} if for all $p \in N$ any $k$ eigenvalues of $\mathrm{Hess}_p\, f$ sum to a non-negative value.
        \item $f$ is \textbf{mean convex} if $f$ is $n$-mean convex.
        \item $f$ is \textbf{$k$-convex} if for all $p\in N$ there are $k$ non-negative eigenvalues of $\mathrm{Hess}_p\, f$.
        \item $f$ is \textbf{convex} if it is $n$-convex or equivalently $1$-mean convex.
    \end{enumerate}
    The adverbs \textbf{strongly} and \textbf{strictly} in front of the above convexity notions mean that the defining inequality $\geq0$ is replaced by $\geq C>0$ or $>0$ respectively. 
\end{defn}

An illustrative example is given by $f(x,y,z) = x^2 + y^2 - \alpha z^2$. It is always $2$-convex, for $\alpha \leq 0$ it is convex, for $\alpha \leq 2$ it is mean convex and for $\alpha \leq 1$ the function $f$ is $2$-mean convex. Interestingly, $f$ is always pullvex with respect to the map $u(x,y,z)=(x,y,\tfrac{1}{\alpha}z)$ for $\alpha\neq 0$. In Figure \ref{fig:exconvex} the level sets of $f(x,y,z)=1$ for $\alpha = -1,0,1$ are visualised.


\begin{figure}[ht]
    \centering
    \includegraphics{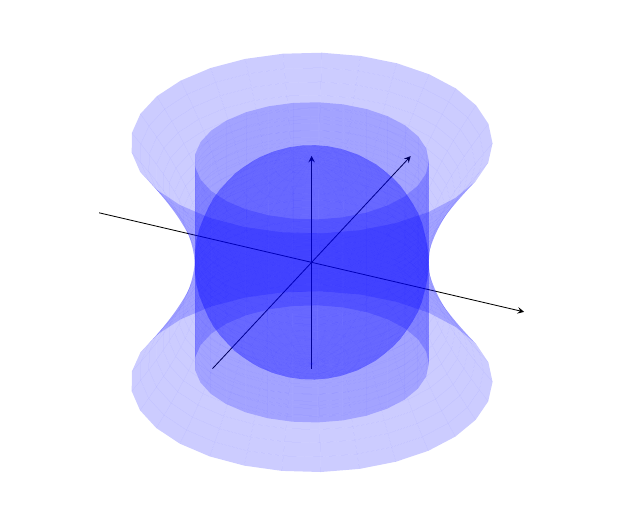}
    \caption{Convexities visualised}
    \label{fig:exconvex}
\end{figure}


\paragraph{Recovering convexity notions} Let $f:(N,h) \rightarrow \R$ be a smooth function.
\begin{enumerate}
    \item Let $\mathcal{F}$ be the set of all isometrically immersed $k$-dimensional manifolds $\iota:S \rightarrow N$. Then pullvexity of a function $u$ means that
    $$
    \mathrm{tr}_{g}(\iota^*\mathrm{Hess} \, f) = \sum_{i=1}^k \lambda_i^2 \mathrm{Hess}\,f(w_i,w_i) = \sum_{i=1}^k \mathrm{Hess}\,f(w_i,w_i)  \geq 0
    $$
    where $\lambda_i=1$ are the singular values of $\iota$. Thus, we obtain the $k$-mean convexity of $f$.
    \item Pullvexity with respect to the identity map $\mathrm{Id}:N \rightarrow N$ recovers the mean convexity $\Delta_N f = \mathrm{tr}_h(\mathrm{Id}^*\mathrm{Hess}\, f) \geq 0$. 
    \item  For a $k$-convex function $f$ take the family $\mathcal{F}$ of curves $\gamma:(-\varepsilon,\varepsilon)\rightarrow N$ with $\mathrm{Hess}\,f(\dot \gamma,\dot\gamma)>0$. Then pullvexity with respect to $\mathcal{F}$ is the same as the $k$-convexity of $f$. Contrary to the case of $k$-mean convexity, where the same family $\mathcal{F}$ can be used for any $k$-mean convex function, the case of $k$-convexity needs a separate family $\mathcal{F}$ for each function $f$.
    \item Let $\mathcal{F}$ be the collection of all immersed curves $\gamma:I \rightarrow N$ defined on an open interval $I$. Then pullvexity with respect to $\mathcal{F}$ means that
    $$
    \mathrm{tr}_g(\gamma^*\mathrm{Hess}\, f) \geq 0
    $$
    which is equivalent to requiring that every eigenvalue of $\mathrm{Hess}\, f$ is non-negative. Thus $f$ is pullvex with respect to $\mathcal{F}$ iff $f$ is convex.
\end{enumerate}

\paragraph{A remark on pullback convexity}

The pullvexity condition $\mathrm{tr}_gu^*\mathrm{Hess}\, f \geq 0$ seems to only investigate the mean convexity of the pulled back Hessian. Thankfully, this is not the case. Indeed, by taking the family $$\mathcal{F} = \big\{ u \circ \gamma \, \mid\, \gamma: ((-\varepsilon,\varepsilon),\gamma^*g) \rightarrow M \text{ smooth curve} \big\}$$ the pullvexity condition is $\mathrm{tr}_{\gamma^*g}((u \circ \gamma)^*\mathrm{Hess}\, f) = \mathrm{Hess}\, f(du(\dot \gamma),du(\dot \gamma))\geq 0$ and recovers the positive semidefiniteness of $u^*\mathrm{Hess}\, f$. Similarly, instead of all curves, we could also take any $k$-dimensional submanifold of $M$ obtaining the $k$-mean convexity of $u^*\mathrm{Hess}\, f$.

\begin{defn}
    Let $u:M \rightarrow N$ and $f:N \rightarrow \R$ be smooth maps. The function $f$ is \textbf{$k$-mean pullvex} with respect to $u$ if $\mathrm{tr}((u\circ \iota)^* \mathrm{Hess}\, f)\geq 0$ for all isometric immersion $\iota:S\rightarrow M$. The adverbs \textbf{strongly} and \textbf{strictly} are used as in the case of pullvexity.
\end{defn}
Clearly, the local condition for $k$-mean pullvexity is
$$
\sum_{i=1}^k \mathrm{Hess}\, f (du(e_i),du(e_i)) \geq 0
$$
for any orthonormal $k$-frame $e_1,\dots,e_k$ of $TM$.

\subsection{Hypersurfaces and pullback convexity}
Now we will investigate when a level set of a function should be regarded as pullvex. 
\paragraph{Pullvex hypersurfaces as level sets} As in the classical case of convex functions with their level sets being convex hypersurfaces, we define the analogous notion for pullvex functions. 
\begin{defn}\label{def:pullvexhypersurface}
    Let $S=f^{-1}(r)$ be a hypersurface defined as a level set of the smooth function $f:N \rightarrow \R$ for the regular value $r\in \R$. The hypersurface $S$ is \textbf{strongly pullvex} with respect to the map $u:M \rightarrow N$ if for each $p\in M$ with $u(p)\in S$ the pullvexity condition $$\mathrm{tr}_g(u^*\mathrm{Hess}\,f)_p \geq C >0$$ holds. We call $f^{-1}(-\infty,r)$ the \textbf{pullvex side} of $S$. $S$ is \textbf{strongly pullvex} with respect to the family of maps $\mathcal{F}$ if the function $f$ satisfies the pullvexity condition for every member $u$ of $\mathcal{F}$.
\end{defn}
Note, that in the definition of a pullvex hypersurface, the pullvexity condition is a priori only satisfied on $u(M)\cap S$ but also possesses information about the normal directions of $TS\subseteq TN$. Let $p \in M$ with $u(p)\in S$. By continuity of $\mathrm{tr}_g(u^*\mathrm{Hess}\,f)$, the pullvexity condition is still true for a geodesic ball $B_\varepsilon(p)\subseteq M$ with a possibly smaller constant $C\geq C'>0$. Of course, by taking a pullvex function $f$ and a regular value $r$ then $S = f^{-1}(r)$ defines a pullvex hypersurface. Moreover, a pullvex function also defines a foliation of $N$ by pullvex hypersurfaces (except for the singular values).

\paragraph{Pullvex hypersurfaces and their second fundamental forms}

Let $f:N \rightarrow \R$ be a smooth function and define for a regular value $r\in \R$ the hypersurface $S = f^{-1}(r)$. By definition $\mathrm{grad}\, f$ is a normal vector field along $S$. The second fundamental form is $A(X,Y) = h(\nabla_XY, N_f)$ for the unit normal field $N_f = \pm \frac{\mathrm{grad}\, f}{\|\mathrm{grad}\, f\|}$, where the sign $\pm$ depends on the choice of orientation. 

\begin{lemma}\label{lem:hypersurfacepullvex}
    Let $f:N \rightarrow \R$ be the defining function of the hypersurface $S = f^{-1}(r)$. Then the relation between the second fundamental form $A$ and the Hessian of $f$ is given by $$ A = \pm \frac{1}{\|\mathrm{grad}\, f\|}\iota^*\mathrm{Hess}\,f, $$ where $\iota: S \rightarrow N$ is the embedding and the sign $\pm$ depends on the choice of the unit normal $N_f$.
\end{lemma}
\begin{proof}
    Since $A(X,Y) = h(\nabla_XY, N_f) = -h (\nabla_XN_f,Y)$ by the orthogonality of tangent vectors $X,Y$ and the unit normal $N_f =\frac{\mathrm{grad}\, f}{\|\mathrm{grad}\, f\|}$, we need to compute
    $$
    \nabla_X N_f = \nabla_X\|\mathrm{grad}\, f\|^{-1} \mathrm{grad}\, f = - \frac{h(\nabla_X \mathrm{grad}\, f,N_f)}{\|\mathrm{grad}\, f\|} \, N_f  + \frac{\nabla_X\mathrm{grad}\, f}{\|\mathrm{grad}\, f\|}
    $$
    and hence
    $$
    A(X,Y) = \frac{1}{\|\mathrm{grad}\, f\|}\bigg( \mathrm{Hess}\, f (X,N_f)h(N_f,Y) - \mathrm{Hess}\, f(X,Y) \bigg) = \frac{1}{\|\mathrm{grad}\, f\|}\mathrm{Hess}\, f(X,Y).
    $$
\end{proof}

By Lemma \ref{lem:hypersurfacepullvex}, requiring conditions on the second fundamental form is a priori not enough for uniquely determining the kind of pullvexity of the hypersurface. For example, the unit sphere $S^2$ in $\R^3$ can be either defined as the level set of $x^2+y^2+z^2$ or of $\sqrt{x^2+y^2+z^2}$ for the regular value $1$. While the former is convex, the latter is $2$-convex. We will see later, that this distinction allows us to prove slightly different results.

\section{The maximum principle and its applications}

\subsection{The maximum principle}

The first goal of this section is establishing a maximum principle for harmonic maps, with pullvexity and the Omori-Yau maximum principle as the major ingredients. As a first step, we will recall the Omori-Yau maximum principle.

\paragraph{Maximum principles on non-compact manifolds}

The definition of the Omori-Yau maximum principle takes the conclusions of the maximum principles of \cite{Omori} and \cite{Yau} as a definition for a Riemannian manifold.

\begin{defn}\label{def:OY}
    A Riemannian manifold $(M,g)$ is said to satisfy the \textbf{strong Omori-Yau maximum principle} if for any function $f\in \mathscr{C}^2(M)$ bounded from above there is a \textbf{maximising sequence} $\{p_k\}$ satisfying
    \begin{enumerate}
        \item $f(p_k) > \sup f - 1/k$,
        \item $\|\mathrm{grad}_g f(p_k)\| < 1/k$ and
        \item $\Delta_M f(p_k) < 1/k$.
    \end{enumerate}
    If one does not require (b), then one refers to the \textbf{weak Omori-Yau maximum principle}.
\end{defn}
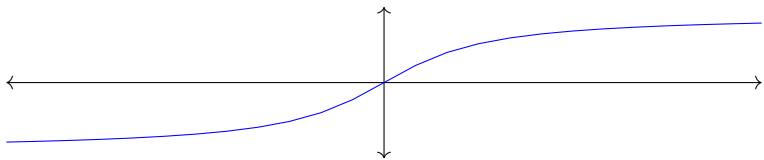
\begin{figure}[h]
    \centering
 \begin{tikzpicture}[domain=-5:5]
 \draw[<->] (-5,0) -- (5,0);
 \draw[<->] (0,-1) -- (0,1);
  \draw[color=blue]   plot (\x, {0.01*atan(\x)}) ;
\end{tikzpicture}
    \caption{Graph of $\arctan(x)$}
    \label{fig:enter-label}
\end{figure}
The conditions of the Omori-Yau maximum principle can be visualised by the example of the graph of $\arctan(x)$. The supremum of the function is $\pi/2$ and by taking any maximising sequence (e.g. $p_k = k$) we notice that the conditions (a), (b) and (c) are satisfied.

\begin{remark}\label{rmk:OYobservation}
    An observation that will be used later is the following perturbation argument. Let $(M,g)$ be a manifold satisfying the strong Omori-Yau maximum principle. Then for any bounded $\mathscr{C}^2(M)$ function $f$, there is a maximising sequence $p_k$ in $M$. A first observation is that we can consider a subsequence of $p_k$. This allows us to disregard points of $M$ on which $f$ has undesirable properties. A second observation is given by the continuity of $f$, $\|\mathrm{grad}_g\,f\|$ and $\Delta f$.  Around every $p_k$, we can choose a geodesic ball of radius $0<\varepsilon_k<1/k$ such that any sequence $q_k$, satisfying $q_k \in B_{\varepsilon_k}(p_k)$, is also a maximising sequence. Hence, as long as non-desirable properties of $f$ do not happen on open sets, they can be ignored.
\end{remark}

\medskip

Some classes of manifolds satisfying the Omori-Yau maximum principle are described in the following theorem.
\begin{thm}[Omori \cite{Omori}, Yau \cite{Yau}, Grigori'yan \cite{Grigoryan}]\label{thm:OmoriYauGrigoriyan}
    Let $(M,g)$ be a complete Riemannian manifold.
    \begin{enumerate}
        \item If the sectional curvature $K_M$ is bounded from below by a uniform constant, then $M$ satisfies the strong Omori-Yau maximum principle. \footnote{Omori actually proved the strong Hessian maximum principle which implies the strong Omori-Yau maximum principle.}
        \item If the Ricci curvature $\mathrm{Ric}_M$ is bounded from below by a uniform constant, then $M$ satisfies the strong Omori-Yau maximum principle.
        \item If the volume growth of geodesic balls is bounded by $\exp(r^2)$ then $M$ satisfies the weak Omori-Yau maximum principle.
    \end{enumerate}
\end{thm}

On compact manifolds, the classical maximum principle holds, hence are a basic class of examples of manifolds satisfying the strong Omori-Yau maximum principle. By Theorem \ref{thm:OmoriYauGrigoriyan} we can construct non-compact examples like the Euclidean space $\R^n$, the hyperbolic space $\mathbb{H}^n$, and manifolds with subexponential volume growth like an asymptotically-locally Euclidean (ALE) four manifold. A classical example of a manifold satisfying the weak Omori-Yau maximum principle but not the strong one is $\R^2\setminus (0,0)$.
\medskip

We would like to point out that in the literature a uniform lower bound on the Ricci curvature is a popular assumption, and by the results cited above, those are examples of manifolds satisfying the strong Omori-Yau maximum principle.

\paragraph{Stochastic completeness}

As a side note, the weak Omori-Yau maximum principle is related to the infinite lifetime of the Brownian motion on the manifold. More precisely, let $\rho_t(p,q)$ be the minimal heat kernel\footnote{For more details see \cite{Hsu}.}, i.e. the fundamental solution to $L = \partial_t - \frac{1}{2}\Delta_M$ with $\Delta_M$ being the Laplace-Beltrami operator on $M$. Then $M$ is \textbf{stochastically complete} if $\int _M\rho_t(x,y)\, dy = 1$. In the influential article \cite{PigolaRigoliSetti}, building on \cite{Grigoryan}, the stochastic completeness of the manifold $(M,g)$ is shown to be equivalent to the validity of the weak Omori-Yau maximum principle. Hence, we will use the stochastic completeness of a manifold as a synonym for the validity of the weak Omori-Yau maximum principle. For more conditions implying stochastic completeness, consult the books \cite{maxprinciples} and \cite{maxprinciplealiasrigoli}.

\paragraph{An Omori-Yau maximum principle for harmonic maps}
The idea of this maximum principle is rather simple. Consider the Sampson maximum principle (cf. \cite{sampson}) and use the Omori-Yau maximum principle rather than the classical maximum principle for the contradiction in the proof of Theorem 2 in \cite{sampson}.

\begin{thm}[Omori-Yau for harmonic maps]\label{thm:HOY}
Let $(M,g)$ be stochastically complete and $u:(M,g) \rightarrow (N,h)$ be a harmonic map.
\begin{enumerate}
    \item For any strongly pullvex function $f:N \rightarrow \R$ the composition $f\circ u$ is subharmonic on $M$ and unbounded from above.
    \item Moreover, if there is a $r < \sup(f \circ u)$, then the pullvexity only needs to hold on $f^{-1}[r,\infty)$.
\end{enumerate}
\end{thm}
\begin{proof}
Suppose $f \circ u$ would be bounded from above by the constant $s = \sup (f \circ u)$. Then the stochastic completeness of $M$ would infer the existence of a  sequence $\{p_k\}_k$ in $M$ such that $(f\circ u)(p_k)> s - 1/k $ and $\Delta_M(f\circ u)(p_k)<1/k$. By the composition formula
$$
\Delta_M(f\circ u) = df(\Delta u) + \mathrm{tr}_g(u^*\mathrm{Hess}\, f)= \mathrm{tr}_g(u^*\mathrm{Hess}\, f)\geq C >0
$$
with $C$ being the constant of strong pullvexity of $f$, we obtain the subharmonicity of $f \circ u$. 

Moreover, combining the two previous observations leads to
$$
0 < C \leq \Delta_M(f\circ u)(p_k)<1/k
$$
and hence to a contradiction. Thus, $f \circ u$ cannot be bounded and (a) holds. For (b) we only need to observe that any maximising sequence $\{p_k\}$ needs to enter $f^{-1}[r,\infty)$.
\end{proof}
Statement (a) should be treated as a global maximum principle and (b) as a localised version in the presence of additional information known about the image of the map $u$. 

\paragraph{Recovering Sampson's maximum principle}

Since the proof of Theorem \ref{thm:HOY} is inspired by the proof of Sampson's maximum principle a natural question would be if the original theorem can be recovered.

\begin{prop}[Pullvex Sampson maximum principle]\label{prop:pullvex_sampson}
    Let $u:M \rightarrow N$ be a non-constant harmonic map and $S\subseteq N$ a strongly pullvex hypersurface with respect to $u$. Assume that there is a touching point, i.e.  $p \in M$ such that $u(p)\in S$. Then the image of $u$ cannot lie on the pullvex side of $S$.
\end{prop}

\begin{proof}
    Let $f:N \rightarrow \R$ be such that $f^{-1}(0)=S$. If $u$ was on the pullvex side of $S$ then $\max (f \circ u) = 0$ which is attained at $p\in M$. As already remarked, the pullvexity condition at $u(p)\in S$ is also valid in a small neighbourhood $U$ of $p\in M$. Thus, a maximising sequence $\{p_k\}$ will eventually approach $p$. By applying Theorem \ref{thm:HOY} we obtain a contradiction.
\end{proof}

The previous proposition also shows that taking a pullvex function $f:N \rightarrow \R$ and defining the leaves $\mathcal{L}_r:= f^{-1}(r)$ for all $r \in \R$ produces a foliation of $N$. Hence, Theorem \ref{thm:HOY} is an analogue of the foliated maximum principle introduced in \cite{assimosJost} and generalised in \cite{ABG}.

\paragraph{Consequences of boundedness}
A first obstruction for the applicability of the Omori-Yau maximum principle for harmonic maps appears whenever boundedness is present. 
\begin{prop}
Let $u:M\rightarrow N$ be a harmonic map. If one of the following properties
\begin{enumerate}
    \item $M$ is compact,
    \item $M$ is stochastically complete, $N$ complete, and $f(M)$ bounded in $N$ 
    \item $M$ is stochastically complete and $N$ is compact
\end{enumerate}
holds, then there is no strongly pullvex function (with respect to $u$) on $N$.
\end{prop}
\begin{proof}
In every case, the existence of a strongly pullvex function would contradict the conclusion of Theorem \ref{thm:HOY}.
\end{proof}
An immediate consequence of the previous proposition is the following corollary.
\begin{cor}
    Let $(N,h)$ be a Riemannian manifold admitting a compact minimal submanifold $u: M\rightarrow N$, or more specifically, a closed geodesic $u:S^1 \rightarrow N$. Then $N$ does not admit any globally defined strongly pullvex function with respect to $u$.
\end{cor}

\subsection{Geometric construction of pullback convex functions}

In this section, we will establish several constructions of strongly pullvex functions that are needed for the validity of Theorem \ref{thm:HOY}.

\paragraph{Conformal maps} Conformal maps are by definition maps that preserve the angle $$\cos(\theta) = \frac{\langle v,w \rangle}{\|v\|\|w\|}$$ between any two vectors $v,w$ and their pushforwards. There are two definitions of conformality which we will use in the subsequent sections.

\begin{defn}
A smooth map $u:(M,g) \rightarrow (N,h)$ between Riemannian manifolds is 
\begin{enumerate}
    \item \textbf{conformal} if $u^*h = e^{2\phi} g$ for a smooth function $\phi:M \rightarrow \R$.
    \item \textbf{horizontally conformal} if $u$ is a submersion and on the horizontal distribution $\mathcal{H} = (\ker du)^\perp$ there is a $\phi:M \rightarrow \R$ such that $h(du(X),du(Y)) = e^{2\phi}g(X, Y)$ for all horizontal vector fields $X,Y$, i.e. sections of $\mathcal{H}$.
\end{enumerate}
 The function $\phi$ is the \textbf{conformal factor} and if there exists a $B>0$ such that $e^{2\phi}\geq B$ then the adverb \textbf{strongly} will be used.
\end{defn}

The class of strongly conformal maps includes several interesting subclasses. For example, an isometric immersion is a conformal immersion with conformal factor $1$. A homothety is a conformal map with a constant conformal factor and a Riemannian submersion is a horizontally conformal map with conformal factor $1$.

\medskip

Conformality and horizontal conformality can also be characterised in terms of their singular values.

\begin{lemma}\label{lem:confpullvex}
Let $u:(M,g)\rightarrow (N,h)$ be a (horizontally) conformal map. The (non-zero) singular values $\lambda_1,\dots,\lambda_{\mathrm{rk}(u)}$ of $u$  are $\lambda_1= \dots = \lambda_{\mathrm{rk}(u)} = e^{\phi}$.
\end{lemma}
\begin{proof}
    Let $v_1,\dots,v_m$ and $w_1,\dots, w_n$ be the singular vectors of $D_pu$. Then $$\lambda_i^2 = \lambda_i^2 h(w_i,w_i) = u^*h(v_i,v_i) = e^{2\phi} g(v_i,v_i) = e^{2\phi}$$ shows that all singular values are $e^{\phi}$.
\end{proof}

By the previous result, a strongly conformal map is always an immersion, hence $\dim M \leq \dim N$. In the case of horizontal conformality the opposite condition $\dim M \geq \dim N$ holds.

\paragraph{Properties at infinity and wide maps}
As a weakening of the conformality condition, we define the following properties at infinity. This allows the formulation of more general statements about harmonic maps.

\begin{defn}
    Let $u:(M,g) \rightarrow (N,h)$ be a smooth map. Denote by $\lambda_1 \geq \dots \geq \lambda_m$ the singular values of $u$.
    \begin{enumerate}    
        \item We call $u$ \textbf{non-singular at infinity} if there is a $c>0$ such that $\lambda_i\geq c$ for all $i=1,\dots,\ell$ with $\ell=\min\{\dim M, \dim N\}$.
        \item We call $u$ \textbf{non-constant at infinity} if there is a $c>0$ such that $e(u) = \|Du\|^2 = \lambda_1^2+\dots +\lambda_m^2\geq c$.
    \end{enumerate}
\end{defn}

A map $u$ that is non-singular at infinity is automatically an immersion or a submersion, depending on the dimensions of $M$ and $N$. Thus, $\mathrm{rk}\, u = \min\{\dim M, \dim N\}$ whenever $u$ is non-singular at infinity. A map that is non-singular at infinity is automatically non-constant at infinity, but the converse does not hold.

Typical examples of maps being non-singular at infinity are strongly conformal and strongly horizontally conformal maps.

\medskip

\begin{defn}
    Let $P:(N,h) \rightarrow (B,\langle \cdot, \cdot \rangle)$ be a smooth map between Riemannian manifolds. $P$ is called \textbf{$k$-wide} if for every local $k$-frame $w_1,\dots, w_k$ of $TN$ the sum
    $$
    \sum_{i=1}^k \langle dP(w_i), dP(w_i) \rangle \geq c>0
    $$
    is bounded from below by some constant $c$.
\end{defn}

The previous definition is a kind of pullvexity assumption for the Riemannian metric $\langle \cdot, \cdot \rangle$ on $B$.

\begin{lemma}\label{lemma:wideconformal}
    Let $P:(N,h) \rightarrow (B,\langle \cdot, \cdot \rangle)$ be a strongly horizontally conformal map. If there exists a $k>0$ with $k + \mathrm{rk}\, P - \dim N \geq 1$, then $P$ is $k$-wide.
\end{lemma}
\begin{proof}
    Denote by $L$ the horizontal projection of $TN$ onto the horizontal distribution $\mathcal{H} = (\ker dP)^\perp$. For any orthonormal $k$-frame $w_1,\dots, w_k$ of $TN$ we have 
    $$
    \sum_{i=1}^k \langle dP(w_i), dP(w_i) \rangle = \sum_{i=1}^k \langle dP(L(w_i)), dP(L(w_i)) \rangle = e^{2\phi} \sum_{i=1}^k h(L(w_i), L(w_i)).
    $$
    The result follows by estimating $\sum_{i=1}^k h(L(w_i), L(w_i))$ from below. This is achieved with a basic linear algebra computation, resulting in the minimum of $\sum_{i=1}^k h(L(w_i), L(w_i))$ over all orthonormal $k$-frames being bounded from below by $k + \mathrm{rk}\, P - \dim N \geq 1$.
\end{proof}

\paragraph{The Tomography theorem}
The original idea of tomography, i.e. the concept of projecting an object to a lower dimensional subspace, motivates the theorem below. 
Our idea is projecting the image of a harmonic map via a totally geodesic map and inspect the image of the composition via pullvexity. As a preliminary lemma, we will consider the special case when the rank of the composition is annihilated. 

\begin{lemma}
    Let $u:M\rightarrow N$ be harmonic and $P:N \rightarrow B$ totally geodesic maps. If $P\circ u$ is of rank 0 in an open set, then $u$ maps into a single fiber of $P$.
\end{lemma}
\begin{proof}
    By Proposition \ref{prop:compositionformula} (a), $\Delta(P \circ u) = dP(\Delta u) + \mathrm{tr}_g(u^*\nabla dP) = 0$ the composition is harmonic. Now, by Theorem 1 in \cite{sampson}, the property of being of rank $\mathrm{rk}(P \circ u) = 0$ on an open set extends to the whole connected component. Since $M$ is connected, this means that $P \circ u \equiv b \in B$, i.e. $u(M) \subseteq P^{-1}(b)$.
\end{proof}

We have seen that projecting on an open set to a constant already implies that the image of the harmonic map lies inside a fiber of the totally geodesic map. The Tomography theorem covers cases where the harmonic map is not annihilated by totally geodesic maps.

\begin{thm}[Tomography theorem]\label{thm:tomography}
    Let $(M,g)$ be stochastically complete and consider the smooth maps $u:(M,g) \rightarrow (N,h)$, $P:(N,h) \rightarrow (B,\langle \cdot, \cdot \rangle)$ and $\eta:(B,\langle \cdot, \cdot \rangle)\rightarrow \R$ between Riemannian manifolds. If one of the cases
    \begin{enumerate}
    \item \begin{enumerate}[label=(\roman*)]
            \item $u$ is harmonic and non-singular at infinity,
            \item either $P$ is totally geodesic and $k$-wide with $k \geq \mathrm{rk}\,u$ or 
            \item[(ii')] $P$ is totally geodesic, strongly horizontally conformal with the rank condition $\mathrm{rk}\, u + \mathrm{rk}\, P - \dim N \geq 1$ and
            \item $\eta$ is strongly convex
        \end{enumerate}
    \item \begin{enumerate}[label=(\roman*)]
            \item $u$ is harmonic and strongly (horizontally) conformal,
            \item $P$ is totally geodesic and
            \item $\eta$ is strongly $k$-mean pullvex with respect to $P$ for $k \geq \mathrm{rk}\,u$
        \end{enumerate}
    \item \begin{enumerate}[label=(\roman*)]
            \item $u$ is harmonic and strongly (horizontally) conformal,
            \item $P$ is totally geodesic and strongly horizontally conformal,
            \item the dimensional condition $\mathrm{rk}\, u + \mathrm{rk}\, P - \dim N \geq k$ and
            \item $\eta$ is strongly $k$-mean convex
        \end{enumerate}
    \end{enumerate}
    holds, then $f:= \eta \circ P$ is strongly pullvex with respect to $u$ and $f\circ u$ is unbounded.
\end{thm}

\begin{proof}
    We will use the following notation: 
    \begin{itemize}
        \item $v_1,\dots,v_m$ and $w_1,\dots, w_n$ will denote the singular vectors of $u$ for the cases (a), (b) and for (c) for the map $v:=P\circ u$.
        \item $c_1>0$ will denote the lower bound on the singular values of $u$ in (a). In the cases (b) and (c) it is the lower bound of the conformal factor $e^{\phi}$ of $u$.
        \item In (a), $c_2$ denotes the lower bound of $k$-wideness and in case (c) it is the lower bound of the conformal factor $e^\psi$ of $P$.
        \item Lastly, $c_3$ denotes the lower bound for the strong pullvexity of $\eta$.
    \end{itemize}

    By assuming $f$ to be strongly pullvex with respect to $u$, the second conclusion is just an application of Theorem \ref{thm:HOY} (a). This means that we need to prove the strong pullvexity of $f= \eta \circ P$ in each of the cases.

    \medskip
    
    By $P$ being totally geodesic, $\Delta(P \circ u) = dP(\Delta u) + \mathrm{tr}_g(u^*\Delta dP) = 0$, i.e. the composition $v=P \circ u$ is harmonic. Similarly, $\mathrm{Hess}\, f = d\eta(\nabla d P) + P^*\mathrm{Hess}\, \eta = P^*\mathrm{Hess}\, \eta$ holds by Proposition \ref{prop:compositionformula}.
    
    \medskip
    We begin with case (a) (i), (ii) and (iii). The following estimates
    \begin{align*}
        \Delta_M(f \circ u) &= \sum_{i=1}^m \, \mathrm{Hess}\, f(du(v_i),du(v_i)) \geq c_1^2 \sum_{i=1}^{\mathrm{rk}\, u} \, \mathrm{Hess}\, f(w_i,w_i) 
        \\ &= c_1^2 \sum_{i=1}^{\mathrm{rk}\, u} \, \mathrm{Hess}\, \eta(dP(w_i),dP(w_i)) \geq c_1^2c_3 \sum_{i=1}^{\mathrm{rk}\, u}\, \langle dP(w_i),dP(w_i)\rangle \geq c_1^2c_2c_3,
    \end{align*}
    show the strong pullvexity of $f$. By lemma \ref{lemma:wideconformal} the wideness condition (ii) can be replaced by the stronger dimensional assumption (ii') for concluding the same result.

    \medskip

    In case (b), we obtain
    \begin{align*}
        \Delta_M(f \circ u) &= \sum_{i=1}^m \, \mathrm{Hess}\, f(du(v_i),du(v_i)) = e^{2\phi} \sum_{i=1}^{\mathrm{rk}\, u} \, \mathrm{Hess}\, f(w_i,w_i)
        \\ & = e^{2\phi} \sum_{i=1}^{\mathrm{rk}\, u} \, P^*\mathrm{Hess}\, \eta(w_i,w_i) = e^{2\phi} c_3 \geq c_1^2c_3
    \end{align*}
    proving the strong pullvexity of $f$.
    \medskip

    In case (c), the dimensional assumption implies the rank condition $rk(P \circ u) \geq \mathrm{rk}\, u + \mathrm{rk}\, P  - \dim N \geq k$. The non-zero singular values of $v:=P\circ u$ are $e^{\phi+\psi}$. The computation
    \begin{align*}
        \Delta_M(\eta \circ v) &= e^{2(\phi+\psi)} \sum_{i=1}^{\mathrm{rk}\, u} \, \mathrm{Hess}\, \eta (w_i,w_i) \geq e^{2(\phi+\psi)} c_3 \geq c_1^2c_2^2c_3
    \end{align*}
    shows the strong pullvexity of $f$.
\end{proof}

\begin{remark}\label{rmk:tomography}
    \begin{enumerate}
        \item The properties of $P$ and $\eta$ only need to hold on the images $u(M)$ and $(P\circ u)(M)$ respectively. Thus, by knowing some additional information about $u$, the choice of $P$ and $\eta$ can lead to more general statements.
        \item Even along the images mentioned previously, the properties are only important on an open dense set on $M$, as explained in Remark \ref{rmk:OYobservation}.
        \item The Tomography theorem can be localised, provided we have additional knowledge as in Theorem \ref{thm:HOY}.
        \item Lastly, whenever in the proof the constants are multiplied $c_1c_2c_3$, we could assume that these are strictly positive functions such that their product $c_1c_2c_3 \geq c > 0$ is bounded. This would allow even more flexibility of the definition in certain situations.
    \end{enumerate}
\end{remark}

\subsection{Applications}
Previously, we obtained an Omori-Yau maximum principle and constructed strongly pullvex functions geometrically. Now, we will explicitly apply the aforementioned theorems.

\paragraph{The Calabi conjectures}

We turn our attention to the Calabi conjectures (a) and (b) in the case of the domain being weak Omori-Yau (stochastically complete).

\begin{thm}[Calabi conjectures]\label{thm:calabi}
Let $M$ be a stochastically complete manifold and $u:M \rightarrow \R^n$ a minimal immersion.
    \begin{enumerate}
        \item The image of $u$ is unbounded.
        \item Let $P:\R^n \rightarrow \R^n$ be an orthogonal projection. If $\mathrm{rk}\,u + \mathrm{rk}\, P - n \geq 1$, then $P \circ u$ is unbounded.
    \end{enumerate}
\end{thm}

\begin{proof}
    We will proceed by proving the statements for harmonic maps which are non-singular at infinity. As stated previously, these include minimal immersions as a special case.
    \medskip
    
    Statement (a) follows by Theorem \ref{thm:tomography} (a) with $N=B$, $P=\mathrm{Id}$ and $\eta(x) = d(0,x)^2$ the distance-squared function from $0\in \R^n$.

    \medskip

    For (b), let $N= \R^n$, $B = \mathrm{Im}(P) \cong \R^{\mathrm{rk}\, P}$ and $\eta(x) = d(0,x)^2$ the distance-squared function from $0\in \mathrm{Im}(P)$. Geometrically, $f = \eta \circ P$ measures the square of the distance of $x$ from $\ker P \subseteq \R^n$. Then by Theorem \ref{thm:tomography} (a) the statement follows.
    
\end{proof}

\begin{remark}
    \begin{enumerate}
        \item Note that Theorem \ref{thm:calabi} (b) proves the “more ambitious” Calabi conjecture for stochastically complete minimal hypersurfaces in $\R^n$ for $n \geq 4$ and projections of rank $(n-2)$ (see Conjecture 0.2 in \cite{ColdingMinicozzi}). 
        \item Theorem \ref{thm:calabi} shows that the example constructed by Nadirashvili cannot be stochastically complete, since any projection of rank $2$ would project the example to a bounded subset.
        \item The dimensional assumption $\dim M + \mathrm{rk}\, P - n \geq 1$ is not satisfied for $\dim M = 2, \mathrm{rk}\, P = 1$ and $n=3$, i.e. the results of Colding-Minicozzi \cite{ColdingMinicozzi} are not covered by Theorem \ref{thm:calabi}. Interestingly, under the assumption of $M$ satisfying the \textbf{$L^\infty$-Liouville property} \footnote{See Chapter 13 in \cite{Grigoryan}}, i.e. every bounded solution $v$ of $\Delta_Mv = 0$ needs to be constant, more can be proven. Let $d_P^2$ denote the distance-squared function from $\ker P$ in $\R^3$. In this case, by the boundedness of $P\circ u$ we could conclude $d_P^2 \circ u \cong c \in \R$, which would mean that $M$ is contained in an affine displacement of $\ker P$.
        \item Notice that in Theorem \ref{thm:calabi} (b) a wideness assumption would prove a more general result. Thus, there is a naturally arising question:
        \begin{quote}
            Are orthogonal projections of rank $1$ automatically $2$-wide along\footnote{For the $k$-wideness it is sufficient to consider orthonormal vectors tangent to $u(M)$. In this case $P$ is $k$-wide along $M$} $M$, where $M \rightarrow \R^3$ is a stochastically complete minimally immersed/embedded surface?
        \end{quote}
        This would reprove the case of the Calabi conjecture settled by Colding-Minicozzi.
    \end{enumerate}
\end{remark}

\paragraph{Conformal maps}
Conformal harmonic maps, especially minimal maps, count among the nicest harmonic maps, motivating the first applications in this class of maps.

\begin{prop}\label{prop:conformalpullvex}
    Let $(M,g)$ be stochastically complete and $u:(M,g) \rightarrow (N,h)$ a strongly conformal harmonic map of rank $\mathrm{rk}\, u \geq k$. Then for any strongly $k$-mean convex function $f:M \rightarrow \R$ the composition $f \circ u$ is unbounded.
\end{prop}

\begin{proof}
    This follows by Theorem \ref{thm:tomography} (c) with $N=B$ and $P=\mathrm{Id}$.
\end{proof}

Recall that, a level set hypersurface $S \subseteq (N,h)$ of a smooth function $f:N \rightarrow \R$ is $k$-mean convex if $\mathrm{Hess}_p\, f$ is $k$-mean convex for all $p\in S$. Note that this is an adaptation of Definition \ref{def:pullvexhypersurface}. A well-known theorem about mean-convex hypersurfaces and minimal surfaces in $\R^3$ is the following corollary.

\begin{cor}
Let $S\subseteq N$ be a $k$-mean convex hypersurface. Then a stochastically complete minimal submanifold $u:M \rightarrow N$ touching $S$ at $q = u(p)\in S$ can not be on the pullvex ($k$-mean convex side) side of $S$.  
\end{cor}
\begin{proof}
    The proof is a combination of Proposition \ref{prop:pullvex_sampson} and Proposition \ref{prop:conformalpullvex}.
\end{proof}

An interesting application of Proposition \ref{prop:conformalpullvex} is by taking quadratic forms on $\R^n$ to be the pullvex functions. Let $\beta$ be a quadratic form of signature $(n_+,n_-,n_0)$ and we investigate the existence of harmonic maps inside $C_\beta:=\beta^{-1}(-\infty,r]$ for $r\in \R$. $C_\beta$ is a strongly $k$-mean convex cone if $\beta$ itself is strongly $k$-mean convex.

\begin{cor}\label{cor:coneharmonic}
    Let $u:M \rightarrow \R^n$ be a strongly conformal harmonic immersion of a stochastically complete manifold $M$ of dimension $m$. Then the image of $u$ cannot be included in any strongly $k$-mean convex cone for $k\leq m$.
\end{cor}
\begin{proof}
    This is again just a special case of \ref{prop:conformalpullvex}.
\end{proof}

\begin{ex} The following examples demonstrate some explicit applications of Proposition \ref{cor:coneharmonic}.  
\begin{enumerate}
    \item Define $\beta_\alpha(x,y,z) = x^2 + y^2 - \alpha\,z^2$ for $\alpha < 1$ and denote by $C_\alpha(r)=\beta_\alpha^{-1}(r)$. Then $\beta_\alpha$ is always $2$-mean convex and the result states that there is no strongly conformal harmonic immersion of a stochastically complete surface into any of the $C_\alpha(r)$ (see Figure \ref{fig:exconvex}). Even more explicitly, we obtain the classical statement that no isometrically embedded cylinder can be minimal. Note that for the choice $\alpha<0$ is also true for one-dimensional harmonic immersions, i.e. for geodesics, but for $0 \leq \alpha < 1$ this is not true any more since any such $C_\alpha(r)$ possesses geodesics inside.
    \item Choose positive constants $\delta_1,\dots,\delta_{\ell}$ such that $\sum_{i=1}^{\ell} \delta_i<1$ and define the quadratic form $\beta(x_1,\dots,x_n) = -\sum_{i=1}^\ell \delta_i x_i^2 + \sum_{i=\ell + 1}^n x_i^2$. By construction, $\beta$ is $\ell + 1$-mean convex and the previous proposition states that the cone defined by $\beta$ cannot have minimal immersions of stochastically complete as long as the dimension of $M$ is $\dim M \geq \ell + 1$. In the case of $\ell=4$ and $n=8$ the quadratic form defines a cone that reminds of the Simons cone. 
\end{enumerate}
\end{ex}

\paragraph{Cartan-Hadamard manifolds}

Recall that a Cartan-Hadamard manifold $(N,h)$ is a complete simply connected Riemannian manifold with non-positive sectional curvature, i.e. $K_N\leq 0$. Some obvious examples are the Euclidean space $\R^n$ and the hyperbolic space $\mathbb{H}^n$.

\begin{prop}
    Let $(M,g)$ be stochastically complete, $(N,h)$ Cartan-Hadamard with sectional curvature $K_N\leq K<0$, and let $S\subseteq N$ be a closed totally geodesic submanifold. Denote by $d_S^2:N\rightarrow \R$ the distance-squared function $d_S^2(p) = \inf_{q\in S}(d(p,q))^2$ from $S$. Then for any harmonic immersion $u:M \rightarrow N$ which is non-constant at infinity, the composition $d_S^2\circ u$ needs to be unbounded.
\end{prop}
\begin{proof}
     By Theorem 4.1 in \cite{BishopOneill}, the function $d_S^2$ is well-defined. The only adaptation of the proofs is in the form of replacing $K_N<0$ by $K_N\leq K<0$ which produces strong convexity instead of strict convexity as in the original article of Bishop and O'Neill.
    
    \medskip
    Now the estimate
    $$
    \Delta_M(d_S^2\circ u) = \sum_{i=1}^m \lambda_i^2\mathrm{Hess}\,d_S^2(w_i,w_i) \geq c^2 \sum_{i=1}^m \lambda_i^2 = c^2 e(u)
    $$
    holds and an application of Theorem \ref{thm:HOY} proves the statement.
\end{proof}

\begin{ex}
    \begin{enumerate}
        \item There are no minimal submanifolds in a Cartan-Hadamard manifold $(N,h)$ with $K_N\leq K<0$ that are stochastically complete and bounded.
        \item Let $\gamma$ be a complete geodesic in the hyperbolic space $\mathbb{H}^n$ and let $S=\mathrm{Im}\, \gamma$ be the totally geodesic submanifold. Then the previous proposition states that in any tubular neighbourhood $T_R(\gamma) \subseteq \mathbb{H}^n$ of length $R$ there is no stochastically complete minimal submanifold 
    \end{enumerate}
\end{ex}

\paragraph{Cone and wedge theorems}

First, we will recall a definition of perturbed cones appearing in the paper \cite{ABG} and extend it to perturbed wedges.


\begin{figure}[ht]
    \centering
    \includegraphics{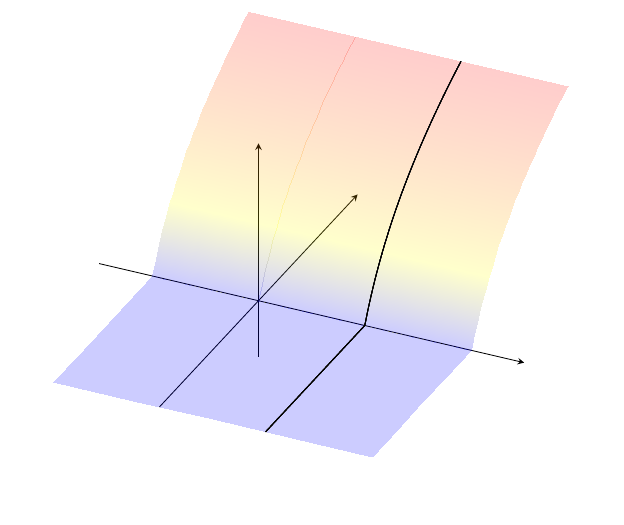}
    \caption{A wedge $C\times \R$ with $C$ defined as the graph of $\log(x+1)$ for $x\geq0$ and $0$ for $x<0$.}
    \label{fig:examplewedge}
\end{figure}


\begin{defn}\label{def:ConeWedge}
    \begin{enumerate}
        \item A connected open set $R\subseteq \R^n$ is said to possess the \textbf{enclosing property} if, for every $p\in R$, there is an affine hyperplane $H$ such that the connected component $B$ of $R\setminus H$ containing $p$ is precompact.
        \item Let $C$ be a closed, path-connected subset of $\R^n$ such that $\R^n\setminus C$ consists of at least two connected components. Let $R_C$ be one of the connected components of $\R^n\setminus C$. If $R_C$ satisfies the enclosing property, then $C$ is called a \textbf{perturbed cone} in Euclidean space and $R_C$ a \textbf{cone region}.
        \item Let $n,k \in \mathbb{N}$ with $n > k$ and $C$ a perturbed cone in $\R^{k}$. A \textbf{perturbed wedge of type $(n,k)$} is $W:= C \times \R^{n-k}$ and is said to have $(n-k)$ flat directions. The \textbf{wedge region} is $R_W = R_C \times \R^{n-k}$ where $R_C$ is a cone region of the cone $C$.
    \end{enumerate}
\end{defn}

An example is visualised in Figure \ref{fig:examplewedge}. This perturbed wedge is not contained in any wedge (in the classical sense). For more examples, just take the examples of perturbed cones in \cite{ABG} and cross them with some $\R^{n-k}$.

\medskip

The next theorem was inspired by Theorem 1.5 in \cite{Borbely} and Theorem 3.4 in \cite{ABG}.

\begin{thm}[Perturbed wedge theorem]\label{thm:perturbedwegde}
    Let $(M,g)$ be a stochastically complete manifold and $u:M \rightarrow \R^n$ a harmonic map that is non-singular at infinity. The image of $u$ cannot be contained in a wedge region of a perturbed wedge of type $(n,k)$ with $\dim M + k-n \geq 1$.
\end{thm}
\begin{proof}
    Let us assume that the image of $u$ is contained in a perturbed wedge $W=C\times \R^{n-k}$ of type $(n,k)$ with underlying perturbed cone $C\subseteq \R^k$. Denote by $P$ the orthogonal projection $\R^n \rightarrow \R^n$ with image $\R^k$ and let $d_Q^2(x)= (d(x,Q))^2$ be the square of the distance function in $\R^k$ from the point $Q$, where $Q$ is chosen the following way:
    
    \medskip
    
    \begin{minipage}[b]{0.4\textwidth}
    Take any $p\in M$ and let $q = (P\circ u)(p) \in \R^k$. Denote by $H$ the hyperplane and by $B$ the precompact set with $q\in B$ obtained via the enclosing property. Let $\nu$ be the unit normal of $H$ pointing into $B$ and define $Q= q - d\nu$. Choose $d$ such that $H\cap\partial B \subseteq B_d(Q)$ but $B \not\subseteq B_d(Q)$. Let $\chi$ be a bump function which is $1$ on the closure of $B$ and falls off rapidly to $0$ in an $\varepsilon$-neighbourhood $B$. By defining $\eta = \chi d_Q^2$ we can apply Theorem \ref{thm:tomography} (a) with Remark \ref{rmk:tomography} (c).
    \end{minipage}
    \begin{minipage}[b]{0.6\textwidth}
    \centering
        \begin{tikzpicture}[scale=0.7]
			\draw [blue,dashed,thick,domain=90:450] plot ({2*cos(\x)}, {2*sin(\x)}) node[anchor=south]{$\partial B_d(Q)$};
			\draw [-,thick] (-3, -1)  -- (-1, -1.72) -- (0,-4) -- (1, -3) -- (1.5,-5) -- (3,1)  -- (2,0) -- (3,3) node[anchor=west]{$C$};
            \draw [-,dotted,thick]  (2,0) -- (-1, -1.72) node[above]{$H$};
            \filldraw[black] (1,-1.72) circle (2 pt) node[anchor=west]{$q$};
            \filldraw[black] (0,0) circle (2 pt) node[anchor=west]{$Q$};
    \end{tikzpicture}
    
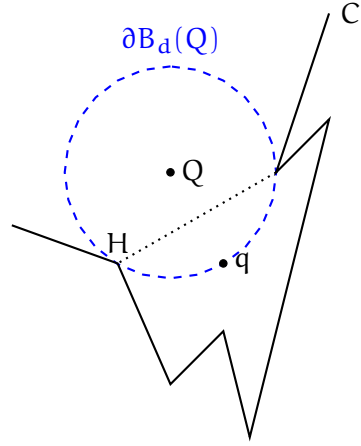
\captionof{figure}{The choice of $Q$.}
    \end{minipage}
\end{proof}
\begin{remark}
    The Perturbed wedge theorem is related to several cone and wedge-like theorems in the literature
    \begin{enumerate}
        \item Theorem 3.4 in \cite{ABG} is an analogue of Theorem \ref{thm:perturbedwegde} which states that there are no proper non-constant harmonic maps with their image inside a perturbed cone (i.e. a wedge with $k=0$ flat directions). 
        \item Theorem 1.28 of \cite{maxprinciples} is also a special case, as non-degenerate cones are special cases of wedges.
        \item By considering stochastically complete minimal surfaces for $n=3$ and $k=2$, Theorem 1.5 of \cite{Borbely} is recovered as a special case. 
    \end{enumerate}
\end{remark}

\paragraph{Simplex traps}

Now we will derive a local manifestation of the perturbed wedge theorem. We will work in $\R^n$ equipped with the coordinates $x=(x_1,\dots,x_n)$ and $H$ denotes the hyperplane defined by $\{x\in \R^n\, \mid \, x_n = 0\}$.

\medskip

Recall that an $N$-simplex $\Delta$ in $\R^n$ with $N\leq n$ is a convex hull $$ \Delta = \left\{\sum_{i=0}^N t_ip_i\,\bigg| \, t_0 + \dots + t_n = 1 \text{ and } t_i\geq 0 \text{ for all }i=0,\dots,N\right\} $$ of $N+1$ point $p_0,\dots,p_N$ under the generic condition that the vectors $p_0 - p_1,\dots,p_0-p_N$ are linearly independent. 

\begin{defn}\label{def:simplextrap}
    Let $\Delta$ be an $N=n-2$ simplex inside $H\subseteq \R^n$, $R\in (0,\infty)$ and choose a unit normal vector $\nu$ for $\Delta \subseteq H$.
    A \textbf{simplex trap with base in $H$} is the set $$T = \left(\bigcup_{s,t \in [0,R]} \Delta_{s,t}\right) \setminus (\Delta_{R,R})_R$$
    where $\Delta_{s,t}:= \Delta + s \nu + t e_n$, $e_n = (0,\dots, 0,1)$ and $(\Delta_{R,R})_R = \{x\in \R^n\, \mid \, d(\Delta_{R,R},x)\leq R\}$ is the closed $R$-neighbourhood of $\Delta_{R,R}$. The boundary of $T$ is the topological boundary $\partial T$ without the boundary of $(\Delta_{R,R})_R$.  A \textbf{simplex trap} $T$ is a rotated and translated simplex trap with a base in $H$.
\end{defn}

\begin{figure}[ht]
    \centering
    \includegraphics{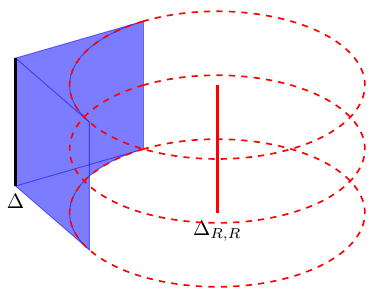}
    \caption{A simplex trap $T$ in $\R^3$}
    \label{fig:simplex trap}
\end{figure}

The following theorem should be treated as an interpretation for simplex traps as measuring tools that can be applied to pieces of minimal hypersurfaces. The conclusion will be that a minimal hypersurface touching a simplex trap has to touch its boundary as well. 

\begin{thm}[Simplex trap theorem]\label{thm:simplextrap}
    Let $M$ be a stochastically complete manifold and $u:M \rightarrow \R^n$ a harmonic immersion that is non-singular at infinity with $\dim M = n-1$. Any simplex trap intersecting $u(M)$ needs to intersect the boundary of $T$ as well.
\end{thm}
\begin{proof}
    We will prove the statement for the simplex trap with base in $H$, i.e. the data $\Delta$, $R\in (0,\infty)$ and a unit normal $\nu$ is given. Suppose there is $p\in M$ with $u(p)\in T$ but the image $u(M)$ does not intersect an $\varepsilon$-neighbourhood of the boundary of $T$. We can take the distance-squared function $d^2_\Delta$ from $\Delta_{R,R}$ and multiply it with the bump function $\chi$ which is constant $1$ on $B = \bigcup_{s,t \in [0,R]} \Delta_{s,t}$ and is $0$ outside an $\varepsilon/2$-neighbourhood of $B$, i.e. we define $f = \chi  d_\Delta^2$. Note that inside $T$ the localised distance-squared function $f$ is strongly $(n-1)$-mean convex. By construction, the maximum of $f$ in $T$ is at $\Delta_{0,0}$, i.e. it is bounded from above by $2R^2$. Moreover, by assumption $f(u(p)) > R^2$. Thus, an application of Theorem \ref{thm:tomography} (a) and Remark \ref{rmk:tomography} (c) produces a contradiction to the boundedness of $f \circ u$ and $u(M)$ needs to intersect the $\varepsilon$-neighbourhood of the boundary of $T$. Sending $\varepsilon \rightarrow 0$ proves the theorem.
\end{proof}
As a simple application, we can deduce that a minimal surface in $\R^3$ cannot have a piece that looks like the region of the minimum of the paraboloid. Of course, this can be computed as the paraboloid would have two positive principal curvatures at the point, but the proof via a simplex trap does not need any computation. 

\begin{remark}
    It is interesting to note that a simplex trap is a metric definition, as it only requires simplices and the ability to measure lengths. Therefore, an extension to Riemannian manifolds could be a viable option.
\end{remark}

\subsection{Parabolicity and halfspaces}

In this section, we will consider a simple combination of pullvexity with the concept of parabolicity.
 
\paragraph{Parabolic manifolds and recurrence}
Parabolicity arose in potential theory and was applied extensively to (sub)harmonic functions. The following definition is classical.
\begin{defn}
    A Riemannian manifold $(M,g)$ is said to be \textbf{parabolic} if any smooth function $u:M \rightarrow \R$ bounded from above satisfying $\Delta u \geq 0$ is constant.
\end{defn}
Some basic examples of parabolic manifolds are $\R,\R^2$ and any compact manifold. It can be shown that any complete Riemannian manifold with volume growth of geodesic balls bounded by $r^2$ is parabolic. The fact that $\R^3$ is non-parabolic already shows that parabolicity is a very restrictive property.

\medskip

As the weak Omori-Yau maximum principle is related to stochastic completeness, there is also a probabilistic interpretation of parabolicity. Namely, the Brownian motion $B_t:(\Omega,\mathcal{F},\mathbb{P}) \rightarrow M$ being recurrent, i.e. for any open set $U \subseteq M$ the probability of almost every trajectory of the Brownian motion $B_t$ starting at the point $p\in M$ reaching $U$ is $1$. In formulas, this can be written as $$\mathbb{P}_p(\{\omega\in \Omega \, \mid \, B_{t_*}(\omega)\in U \text{ for some }t_*>0  \})=1.$$ A valuable resource containing more on parabolicity is the article \cite{Grigoryan}.

\paragraph{Halfspace results for parabolic manifolds}
The following technical lemma states that bounded parabolic manifolds are contained in level sets of pullvex functions.

\begin{lemma}[Level set lemma]
Let $u:M\rightarrow N$ be harmonic and $f:N \rightarrow \R^k$ be a map with its component functions $f_1,\dots, f_k$ being pullvex with respect to $u$. If $f\circ u$ is bounded, then $u(M)$ is contained in $\bigcap_{i=1}^kf^{-1}_i(c_i)$ for $c_i\in \R$.
\end{lemma}
\begin{proof}
The composition formula for the Laplacian implies
$$
\Delta_M( f_i \circ u) = \mathrm{tr}_g(u^* \mathrm{Hess}\, f_i)\geq 0
$$
and hence $f_i \circ u$ is subharmonic. By the boundedness assumption $f_i \circ u \equiv c_i\in \R$ and hence $u(M) \subseteq f_i^{-1}(c)$. Since this is true for all $i$, we obtain $u(M) \subseteq \bigcap_{i=1}^kf^{-1}_i(c_i)$.
\end{proof}
Next, we prove a halfspace result for parabolic manifolds.
\begin{prop}\label{prop:parabolichalfspace}
    Let $(M,g)$ be a parabolic manifold and let $u:M \rightarrow \R^n$ be a harmonic map. If the image of $u$ is contained in a halfspace, then it is contained in a hyperplane parallel to the boundary of the halfspace. 
\end{prop}
\begin{proof}
    Let us assume that $u$ takes its values in the upper halfspace $H_n = \R^{n-1}\times (0,\infty)$. The function $f(x) = \frac{1}{x_n+1}$ is bounded on $H_n$ by $1$ and its Hessian is positive semidefinite with $n-1$ eigenvalues $0$ and an eigenvalue $\frac{2}{(x_n+1)^3}$. Then $f \circ u$ is bounded and by the previous lemma $u(M)\subseteq f^{-1}(r)$ for $r\in (0,1)$, i.e. $u(M)$ is contained in a hyperplane parallel to the boundary of the halfspace $\R^{n-1} \times \{0\}$.
\end{proof}
As a basic example, we can consider the harmonic map $u:\R^2 \rightarrow \R^5$ defined by $(s,t) \mapsto (s,t,s^2-t^2,1,1)$. Then Proposition \ref{prop:parabolichalfspace} applies twice.
\paragraph{A question}
We have seen that, in the case of parabolic manifolds, a halfspace result holds for harmonic maps into $\R^n$. Under the assumption of stochastic completeness we could prove a perturbed wedge theorem. Is there a geometric or probabilistic notion on a manifold for which a halfspace result is true?. Perhaps it should be a notion that is more specialised than stochastic completeness but should include parabolic manifolds as an example.

\bibliography{SampsonOmoriYau}
\addcontentsline{toc}{section}{\bibname}
\end{document}